\documentclass[3p,10pt,a4paper,twoside,fleqn,sort&compress]{filomat}
\usepackage{amssymb,amsmath,latexsym}
\usepackage[varg]{pxfonts}

\newtheorem{theorem}{Theorem}[section]

\newtheorem{definition}[theorem]{Definition}

\newtheorem{lemma}[theorem]{Lemma}

\newtheorem{proposition}[theorem]{Proposition}
\newtheorem{remark}[theorem]{Remark}

\numberwithin{equation}{section}
\newcommand{\FilVol}{xx}
\newcommand{\FilYear}{yyyy}
\newcommand{\FilPages}{zzz--zzz}
\newcommand{\FilDOI}{(will be added later)}
\newcommand{\FilReceived}{Received: dd Month yyyy; Accepted: dd Month yyyy}

\begin{document}

%

\title{$p(x)$-Laplacian-Like Neumann Problems in Variable-Exponent Sobolev Spaces Via Topological Degree Methods}

\author[affil1]{ Mohamed El Ouaarabi}
\ead{mohamedelouaarabi93@gmail.com}
\author[affil1]{Chakir Allalou}
\ead{chakir.allalou@yahoo.fr}
\author[affil1]{Said Melliani}
\ead{s.melliani@usms.ma}
\address[affil1]{Laboratory LMACS, Faculty of Science and Technology, Sultan Moulay Slimane University, Beni Mellal, BP 523, 23000, Morocco.}
\newcommand{\AuthorNames}{Mohamed El Ouaarabi et al.}

\newcommand{\FilMSC}{Primary 35J60; Secondary 35D30, 47H11, 46E35.}
\newcommand{\FilKeywords}{$p(x)$-Laplacian-like operators,  variable-exponent Sobolev spaces, capillarity phenomena, topological degree methods.}
\newcommand{\FilCommunicated}{Maria Alessandra Ragusa}

\begin{abstract}
In this paper, we investigate the existence of a "weak solutions" for a Neumann problems of $p(x)$-Laplacian-like operators, originated from a capillary phenomena, of the following form
\begin{equation*}
	\displaystyle\left\{\begin{array}{ll}
		\displaystyle-{\rm{div}}\Big(\vert\nabla u\vert^{p(x)-2}\nabla u+\frac{\vert\nabla u\vert^{2p(x)-2}\nabla u}{\sqrt{1+\vert\nabla u\vert^{2p(x)}}}\Big)=\lambda f(x, u, \nabla u) & \mathrm{i}\mathrm{n}\ \Omega,\\\\
		\Big(\vert\nabla u\vert^{p(x)-2}\nabla u+\frac{\vert\nabla u\vert^{2p(x)-2}\nabla u}{\sqrt{1+\vert\nabla u\vert^{2p(x)}}}\Big)\frac{\partial u}{\partial\eta}=0 & \mathrm{o}\mathrm{n}\ \partial\Omega,
	\end{array}\right.
\end{equation*}
in the setting of the variable-exponent Sobolev spaces $W^{1,p(x)}(\Omega)$, where $\Omega$ is a smooth bounded domain in $\mathbb{R}^{N}$, $p(x)\in C_{+}(\overline{\Omega})$ and $\lambda$ is a real parameter. Based on the topological degree for a class of demicontinuous operators of generalized $(S_{+})$ type and the theory of variable-exponent Sobolev spaces, we obtain a result on the existence of weak solutions to the considered problem.
\end{abstract}

\maketitle

\makeatletter
\renewcommand\@makefnmark%
{\mbox{\textsuperscript{\normalfont\@thefnmark)}}}
\makeatother
\section{Introduction}\label{Intro}
In recent years, partial differential equations with nonlinearities and nonconstant exponents have received a lot of attention. Perhaps the impulse for this comes from the new search field that reflects a new type of physical phenomenon is a class of nonlinear problems with variable exponents. Modeling with classic Lebesgue and Sobolev spaces has been demonstrated to be limited for a number of materials with inhomogeneities. In the subject of fluid mechanics, for example, great emphasis has been paid to the study of electrorological fluids, which have the ability to modify their mechanical properties when exposed to an electric field (see \cite{X,Y,Z,W}). Rajagopal and M. Ruzicka recently developed a very interesting model for these fluids in \cite{ref15} (see also \cite{ref16}), taking into account the delicate interaction between the electric field $E(x)$ and the moving liquid. This type of problem's energy is provided by $\displaystyle \int_{\Omega}\vert \nabla u\vert ^{p(x)}dx$. This type of energy can also be found in elasticity problems \cite{ref22}. The natural energy space in which such problems can be studied is the variable exponent Sobolev space $W^{1,p(x)}(\Omega)$. Other applications relate to image processing \cite{ref1,ref8}, elasticity \cite{ref31}, the flow in porous media \cite{ref4,ref21}, and problems in the calculus of variations involving variational integrals with nonstandard growth \cite{ref31,ref27,ref11}.

Let $\Omega$ be a smooth bounded domain in $\mathbb{R}^{N} (N\geq 2)$, with a Lipschitz boundary denoted by $\partial\Omega$. In this paper we deal with the question of the  existence of a weak solutions for a class of $p(x)$-Laplacian-like Neumann problems, arising from capillarity phenomena, of the following form:
\begin{equation}\label{P}
\displaystyle\left\{\begin{array}{ll}
\displaystyle-{\rm{div}}\Big(\vert\nabla u\vert^{p(x)-2}\nabla u+\frac{\vert\nabla u\vert^{2p(x)-2}\nabla u}{\sqrt{1+\vert\nabla u\vert^{2p(x)}}}\Big)=\lambda f(x, u, \nabla u) & \mathrm{i}\mathrm{n}\ \Omega,\\\\
\Big(\vert\nabla u\vert^{p(x)-2}\nabla u+\frac{\vert\nabla u\vert^{2p(x)-2}\nabla u}{\sqrt{1+\vert\nabla u\vert^{2p(x)}}}\Big)\frac{\partial u}{\partial\eta}=0 & \mathrm{o}\mathrm{n}\ \partial\Omega,
\end{array}\right.
\end{equation}
where $\frac{\partial u}{\partial\eta}$ is the exterior normal derivative, $p(x)\in C_{+}(\overline{\Omega})$ with $p^{-}:=\min\{h(x),\ x\in\overline{\Omega}\}\leq p(x)\leq p^{+}:=\max\{h(x),\ x\in\overline{\Omega}\}<\infty$, $\lambda$ is a real parameter and $f :\Omega\times \mathbb{R}\times \mathbb{R}^{N}\rightarrow \mathbb{R}$ is a Carath\'eodory function satisfying some non-standard growth condition. The expression $f(x, u, \nabla u)$ is often referred to as a convection term. Note that, since the nonlinearity $f$ depends on the gradient $\nabla u$, then the problem \eqref{P} does not have a variational structure, so the variational methods cannot be applied directly.

Capillarity can be briefly explained by considering the effects of two opposing forces: adhesion, i.e. the attractive (or repulsive) force between the molecules of the liquid and those of the container; and cohesion, i.e. the attractive force between the molecules of the liquid. The study of capillary phenomenon has gained some attention recently. This increasing interest is motivated not only by fascination in naturally occurring phenomena such as motion of drops, bubbles, and waves but also its importance in applied fields ranging from industrial and biomedical and pharmaceutical to microfluidic systems. In the context of the study of capillarity phenomena, recently, problem like \eqref{P} has begun to receive more and more attention, for instance \cite{Avci,Ge,Kim,Rodri,Sho,Vet,Zh}.

Let us recall some known results on problem \eqref{P}. For example, W. Ni and J. Serrin \cite{Ni1,Ni2} initiated the study of ground states for equations of the form
\begin{equation}\label{P2}
\displaystyle-{\rm{div}}\Big(\frac{\nabla u}{\sqrt{1+\vert\nabla u\vert^{2}}}\Big)=f(u) \quad \mathrm{i}\mathrm{n}\ \mathbb{R}^{N},
\end{equation}
with very general right hand side $f$. The operator $\displaystyle-{\rm{div}}\Big(\frac{\nabla u}{\sqrt{1+\vert\nabla u\vert^{2}}}\Big)$ is usually denoted as the prescribed mean curvature operator.

Radial (singular) solutions of the problem \eqref{P2} has been studied in the context of the analysis of capillary surfaces for a function $f$ of the form $f(u) = ku$, for $k > 0$ (for more details see \cite{Con,Fin,Joh}). In \cite{Cl}, the authors studied the existence of positive radial solutions of problem \eqref{P2} with $f(u)=\vert u\vert^{q-1}u$, $q>1$.

Obersnel and Omari in \cite{Obe} based on variational and combines critical point theory, the lower and upper solutions method and elliptic regularization, established the existence and multiplicity of positive solutions of the prescribed mean curvature problem
\begin{equation} \label{P3}
\displaystyle\left\{\begin{array}{ll}
\displaystyle-{\rm{div}}\Big(\frac{\nabla u}{\sqrt{1+\vert\nabla u\vert^{2}}}\Big)=\lambda f(x, u) & \mathrm{i}\mathrm{n}\ \Omega,\\\\
u=0 & \mathrm{o}\mathrm{n}\ \partial\Omega,
\end{array}\right.
\end{equation}
where $\lambda>0$ is a parameter and $f :\Omega\times \mathbb{R}\rightarrow \mathbb{R}$ is a Carathéodory function whose potential satisfies a suitable oscillating behavior at zero.

For $\lambda>0$ and $f$ independent of $\nabla u$ with Dirichlet boundary condition, M. Rodrigues \cite{Rodri}, by using Mountain Pass lemma and Fountain theorem, established the existence of non-trivial solutions for \eqref{P}.  

Problems like \eqref{P}, \eqref{P2} and \eqref{P3} play, as is well known, a role in differential geometry and in the theory of relativity.

In the present paper, we will generalize these works. Using a topological degree for a class of demicontinuous operators of generalized $(S_{+})$ type of \cite{BM} and the theory of the variable-exponent Sobolev spaces, we establish some new sufficient conditions under which the problem \eqref{P} possesses a weak solutions $u$ in $W^{1,p(x)}(\Omega)$. To the best of our knowledge, this is the first paper that discusses the $p(x)$-Laplacian-like type problems with convection term and Neumann boundary data via topological degree theory. For more informations about the history of this theory, the reader can refer to \cite{ACA1,ACA2,BM,CC}.

The remainder of the article is organized as follows. In section \ref{Prelimi}, we review some fundamental preliminaries about the functional framework where we will treat our problem. In Section \ref{Degree}, we introduce some classes of operators of generalized $(S_+)$ type, as well as the Berkovits topological degrees required for the proof of our main result. Finaly, in the Section \ref{Result}, we give our basic assumptions, some technical lemmas, and we will state and prove the main result of the paper.
\section{Preliminaries about the functional framework} \label{Prelimi}
In order to deal with the problem \eqref{P}, we need some theory of the variable-exponent Lebesgue-Sobolev spaces $L^{p(x)}(\Omega)$ and $W^{1,p(x)}(\Omega)$. For convenience, we only recall some basic facts with will be used later, we refer to \cite{ref6,ref9,ref13} for more details.

Let $\Omega$ be a smooth bounded domain in $\mathbb{R}^{N} (N\geq 2)$, with a Lipschitz boundary denoted by $\partial\Omega$. Denote
\begin{equation*}
	C_{+}(\overline{\Omega})=\Big\{h:h\in C(\overline{\Omega}) \mbox{ such that } h(x)>1 \mbox{ for any } x\in\overline{\Omega}\Big\}.
\end{equation*}
For any $h\in C_{+}(\overline{\Omega})$, we define
\begin{equation*}
	h^{+}\ :=\max\Big\{h(x),\ x\in\overline{\Omega}\Big\} \ \mbox{ and }\ h^{-}:=\min\Big\{h(x),\ x\in\overline{\Omega}\Big\}.
\end{equation*}
For any $p\in C_{+}(\overline{\Omega})$ we define the variable exponent Lebesgue space
\begin{equation*}
	L^{p(x)}(\Omega)=\Big\{ u:\Omega\rightarrow \mathbb{R}\ \mbox{ is measurable such that }\ \displaystyle \int_{\Omega}\vert u(x)\vert ^{p(x)}dx<+\infty\Big\},
\end{equation*}
with the norm
\begin{equation*}
	\vert u\vert _{p(x)}=\inf\{\lambda>0/\rho_{p(x)}(\frac{u}{\lambda})\leq 1\},
\end{equation*}
where
\begin{equation*}
	\rho_{p(x)}(u)=\int_{\Omega}\vert u(x)\vert ^{p(x)}dx,\ \forall u\in L^{p(x)}(\Omega).
\end{equation*}
\begin{proposition}\cite{ref6}	
	Let $(u_{n})$ and $u\in L^{p(\cdot)}(\Omega)$, then 
	\begin{equation}\label{i2}
	\vert u\vert _{p(x)}<1(resp. =1;>1)\ \Leftrightarrow\ \rho_{p(x)}(u)<1(resp. =1;>1)\,
	\end{equation}
	\begin{equation}\label{i3}
	\vert u\vert _{p(x)}>1\ \Rightarrow\ \vert u\vert _{p(x)}^{p^{-}}\leq\rho_{p(x)}(u)\leq\vert u\vert _{p(x)}^{p^{+}},
	\end{equation}
	\begin{equation}\label{i4}
	\vert u\vert _{p(x)}<1\ \Rightarrow\ \vert u\vert _{p(x)}^{p^{+}}\leq\rho_{p(x)}(u)\leq\vert u\vert _{p(x)}^{p^{-}},
	\end{equation}
	\begin{equation}\label{i5}
	\lim_{n\rightarrow\infty}\vert u_{n}-u\vert _{p(x)}=0\ \Leftrightarrow\ \lim_{n\rightarrow\infty}\rho_{p(x)}(u_{n}-u)=0.
	\end{equation}
\end{proposition}
\begin{remark}
	Notice that, from \eqref{i3} and \eqref{i4}, we can deduce the inequalities
	\begin{equation}\label{i6}
	\vert u\vert _{p(x)}\leq\rho_{p(x)}(u)+1,
	\end{equation}
	\begin{equation}\label{i7}
	\rho_{p(x)}(u)\leq\vert u\vert _{p(x)}^{p^{-}}+\vert u\vert _{p(x)}^{p^{+}}.
	\end{equation}
\end{remark}
\begin{proposition}\cite[Theorem 2.5 and Corollary 2.7]{ref9}	
	The spaces $L^{p(x)}(\Omega)$ is a separable and reflexive Banach spaces.
\end{proposition}
\begin{proposition}\cite[Theorem 2.1]{ref9}	
	The conjugate space of $L^{p(x)}(\Omega)$ is $L^{p'(x)}(\Omega)$ where $\frac{1}{p(x)}+\frac{1}{p'(x)}=1$ for all $ x\in\Omega.$ For any $u\in L^{p(x)}(\Omega)$ and $v \in L^{p'(x)}(\Omega)$, we have the following H\"{o}lder-type inequality
	\begin{equation}\label{i1}
	\vert \int_{\Omega}uv\ dx\vert \leq(\frac{1}{p-}+\frac{1}{p^{'-}})\vert u\vert _{p(x)}\vert v\vert _{p'(x)}\leq 2\vert u\vert _{p(x)}\vert v\vert _{p'(x)}.
	\end{equation}
\end{proposition}
\begin{remark}\label{embedding}
	If $p_{1},\, p_{2}\in C_{+}(\overline{\Omega})$ with $p_{1}(x)\leq p_{2}(x)$ for any $x\in\overline{\Omega}$, then there exists the continuous embedding $L^{p_{2}(x)}(\Omega)\hookrightarrow L^{p_{1}(x)}(\Omega)$.
\end{remark}
Now, we define the variable exponent Sobolev space $W^{1,p(x)}(\Omega)$ as
\begin{equation*}
	W^{1,p(x)}(\Omega)=\Big\{u\in L^{p(x)}(\Omega)\mbox{ such that }\vert \nabla u\vert \in L^{p(x)}(\Omega)\Big\},
\end{equation*}

with the norm
\begin{equation*}
	\vert \vert u\vert \vert =\vert u\vert _{p(x)}+\vert \nabla u\vert _{p(x)}.
\end{equation*}
We denote by $W_{0}^{1,p(\cdot)}(\Omega)$ the closure of $C_{0}^{\infty}(\Omega)$ in $W^{1,p(\cdot)}(\Omega)$.
\begin{proposition}\cite{ref81,ref111}
	If the exponent $p(\cdot)$ satisfies the log-H\"{o}lder continuity condition, i.e. there is a constant $\alpha>0$ such that for every $x,\  y\in\Omega,\ x\neq y$ with $\vert x-y\vert \displaystyle \leq\frac{1}{2}$ one has
	\begin{equation}\label{i8}
	\vert p(x)-p(y)\vert \leq\frac{\alpha}{-\log\vert x-y\vert },
	\end{equation}
	then we have the poincar\'{e} inequality, i.e. the exists a constant $C>0$ depending only on $\Omega$ and the function $p$ such that
	\begin{equation}\label{i9}
	\vert u\vert _{p(x)}\leq C\vert \nabla u\vert _{p(x)},\ \forall\ u\in W_{0}^{1,p(\cdot)}(\Omega).
	\end{equation}
\end{proposition} 
In this paper we will use the following equivalent norm on $W^{1,p(\cdot)}(\Omega)$
\begin{equation*}
	\vert u\vert _{1,p(x)}=\vert \nabla u\vert _{p(x)}, 
\end{equation*}
which is equivalent to $\vert \vert \cdot\vert \vert $.\\
Furthermore, we have the compact embedding $W^{1,p(\cdot)}(\Omega)\hookrightarrow L^{p(\cdot)}(\Omega)$(see \cite{ref9}).
\begin{proposition}\cite{ref6,ref9}
	The spaces $\Big(W^{1,p(x)}(\Omega),\vert\cdot\vert _{1,p(x)}\Big)$ and $\Big(W_{0}^{1,p(x)}(\Omega),\vert\cdot\vert _{1,p(x)}\Big)$ are separable and reflexive Banach spaces.
\end{proposition}
\begin{remark}
	The dual space of $W^{1,p(x)}(\Omega)$ denoted $W^{-1,p'(x)}(\Omega)$, is equipped with the norm
	\begin{equation*}
		\vert u\vert _{-1,p'(x)}=\inf\Big\{\vert u_{0}\vert _{p'(x)}+\sum_{i=1}^{N}\vert u_{i}\vert _{p'(x)}\Big\},
	\end{equation*}
	where the infinimum is taken on all possible decompositions $u=u_{0}-{\rm{div}}F$ with $u_{0}\in L^{p'(x)}(\Omega)$ and \\$F=(u_{1}, \ldots,u_{N})\in(L^{p'(x)}(\Omega))^{N}$.
\end{remark}
\section{A review on some classes of mappings and topological degree theory} \label{Degree}
Now, we give some results and properties from the theory of topological degree. We start by defining some classes of mappings. 

In what follows, let $X$ be a real separable reflexive Banach space and $X^{*}$ be its dual space with dual pairing $\langle \,\cdot\,,\,\cdot\,\rangle $ and given a nonempty subset $\Omega$ of $X$. Strong (weak) convergence is represented by the symbol $\rightarrow(\rightharpoonup)$.
\begin{definition}
	Let $Y$ be another real Banach space. A operator $F:\Omega\subset X\rightarrow Y$ is said to be
	\begin{enumerate}
		\item  bounded, if it takes any bounded set into a bounded set.
		\item  demicontinuous, if for any sequence $(u_{n})\subset\Omega,\; u_{n}\rightarrow u$ implies $F(u_{n})\rightharpoonup F(u)$. 
		\item  compact, if it is continuous and the image of any bounded set is relatively compact.
	\end{enumerate}
\end{definition}
\begin{definition}
	A mapping $F:\Omega\subset X\rightarrow X^{*}$ is said to be
	\begin{enumerate}
		\item of class $(S_{+})$, if for any sequence $(u_{n})\subset\Omega$ with $u_{n}\rightharpoonup u$ and $\displaystyle \limsup_{n\rightarrow \infty}\langle Fu_{n}, u_{n}-u\rangle\leq 0$, we have $u_{n}\rightarrow u.$
		\item quasimonotone, if for any sequence $(u_{n})\subset\Omega$ with $u_{n}\rightharpoonup u$, we have $\displaystyle \limsup_{n\rightarrow \infty}\langle Fu_{n},\ u_{n}-u\rangle\geq 0.$
	\end{enumerate}
\end{definition}
\begin{definition} Let  $T$ : $\Omega_{1}\subset X\rightarrow X^{*}$ be a bounded operator such that $\Omega\subset\Omega_{1}$.
	For any operator $F$ : $\Omega \subset X \rightarrow X$, we say that
	\begin{enumerate}
		\item  $\;F$ of class $\;(S_{+})_{T},$ if for any sequence $(u_{n})\subset \Omega$ with $u_{n}\rightharpoonup u,\; y_{n}:=Tu_{n}\rightharpoonup y$ and $\displaystyle \limsup_{n\rightarrow \infty}\langle Fu_{n},\ y_{n}-y\rangle\leq 0,$ we have $\;u_{n}\rightarrow u$. 
		\item  $\;F\;$ has the property $(QM)_{T}$, if for any sequence $(u_{n})\subset\Omega$ with $u_{n} \rightharpoonup u,\;\, y_{n}:=Tu_{n} \rightharpoonup y$, we have $\displaystyle \limsup_{n\rightarrow \infty}\langle Fu_{n}, y-y_{n}\rangle\geq 0.$
	\end{enumerate}
\end{definition}
In the sequel, we consider the following classes of operators:
\begin{align*}
	&\mathcal{F}_{1} (\Omega) :=\{ F : \Omega\rightarrow X^{*}\backslash F \mbox{ is bounded, demicontinuous }\mbox{and of class } (S_{+}) \},\\\\
	&\mathcal{F}_{B}(X):=\{F\in \mathcal{F}_{T,B}(\overline{E})\backslash\; E\in \mathcal{O},\ \mathrm{T}\in \mathcal{F}_{1}(\overline{\mathrm{E}})\}.\\\\
	&\mathcal{F}_{T,B}(\Omega):=\{ F:\Omega\rightarrow X\backslash F\mbox{ is bounded, demicontinuous } \mbox{and of class } (S_{+})_{T}\},\\\\ 
	&\mathcal{F}_{T}(\Omega):=\{ F:\Omega\rightarrow X\backslash\; F \;\mbox{is demicontinuous and of class } (S_{+})_{T}\},
\end{align*}
for any $\Omega\subset D(F)$, where $D(F)$ denotes the domain of $F$, and any $T\in \mathcal{F}_{1} (\Omega)$.\\ 
Now, let $\mathcal{O}$ be the collection of all bounded open set in $X$ and we define
\begin{align*}
	&\mathcal{F}(X):=\{F\in \mathcal{F}_{T}(\overline{E})\;\vert \; E\in \mathcal{O},\ \mathrm{T}\in \mathcal{F}_{1}(\overline{\mathrm{E}})\},
\end{align*}
where, $\mathrm{T}\in \mathcal{F}_{1}(\overline{\mathrm{E}})$ is called an essential inner map to $F$.
\begin{lemma} \label{lemma1}\cite[Lemma 2.3]{ref7} Let $T\in \mathcal{F}_{1}(\overline{E})$ be continuous and $S$ : $D(S)\subset X^{*}\rightarrow X$ be demicontinuous such that $T(\overline{E})\subset D(S)$, where $E\;$ is a bounded open set in a real reflexive Banach space $X$. Then the following statements are true :
	\begin{enumerate}
		\item If $S$ is quasimonotone, then $I+SoT\in \mathcal{F}_{T}(\overline{E})$, where $I$ denotes the identity operator.
		\item  If $S$ is of class $(S_{+})$, then $SoT\in \mathcal{F}_{T}(\overline{E})$.
	\end{enumerate}
\end{lemma}
\begin{definition}
	Suppose that $E$ is bounded open subset of a real reflexive Banach space $X$, $T\in \mathcal{F}_{1}(\overline{E})$ be continuous and let $F, S\in \mathcal{F}_{T}(\overline{E})$. The affine homotopy
	$\Lambda$ : $[0,1]\times\overline{E}\rightarrow X$ defined by
	$$\Lambda(t,u) :=(1-t)Fu+tSu,\quad \mbox{for} \quad (t,u)\in[0,1]\times\overline{E}$$
	is called an admissible affine homotopy with the common continuous essential inner map $T$.
\end{definition}
\begin{remark} \cite[Lemma 2.5]{ref7} \emph{The above affine homotopy is of class} $(S_{+})_{T}.$
\end{remark}
Next, as in \cite{ref7} we give the topological degree for the class $\mathcal{F}(X)$.

\begin{theorem}\label{thm_dgr} Let
	$$M=\big\{(F,E,h)\backslash E\in \mathcal{O},\;T\in \mathcal{F}_{1}(\overline{E}),\ F\in \mathcal{F}_{T,B}(\overline{E}),\ h\not\in F(\partial E)\big\}.$$
	Then, there exists a unique degree function $d:M\longrightarrow \mathbb{Z}$
	that satisfies the following properties:
	\begin{enumerate}
		\item (\emph{Normalization}) For any $h\in E$, we have $$d(I,E,h)=1.$$
		\item (\emph{Additivity}) Let $F\in \mathcal{F}_{T,B}(\overline{E})$. If $E_{1}$ and $E_{2}$ are two disjoint open subsets of $E$ such that $h\not\in F(\overline{E}\backslash (E_{1}\cup E_{2}))$, then we have
		$$d(F,E,h)=d(F,E_{1},h)+d(F,E_{2},h).$$
		\item (\emph{Homotopy invariance}) If $\Lambda$ : $[0,1]\times\overline{E}\rightarrow X$ is a bounded admissible affine homotopy with a common continuous essential inner map and $h$: $[0,1]\rightarrow X$ is a continuous path in $X$ such that $h(t)\not\in \Lambda(t,\partial E)$ for all $t\in[0,1]$, then $$d(\Lambda(t,\cdot),E,h(t))=const\ {\it for\ all }\ t\in[0,1].$$
		\item (\emph{Existence}) If $d(F,E,h)\neq 0$, then the equation $Fu=h$ has a solution in $E.$
		\item (\emph{ Boundary dependence}) {\it If} $F,\ S\in \mathcal{F}_{\mathrm{T}}(\overline{\mathrm{E}})$ {\it coincide on} $\partial E$ {\it and} $h\not\in F(\partial E)$, {\it then}
		$$ d(F,E, h)=d(S,\ E,\ h)$$
	\end{enumerate}
\end{theorem}
\begin{definition}\cite[ Definition 3.3]{ref7} The above degree is defined as follows:
	$$ d(F, E, h)\ :=d_{B}(F\vert _{\overline{E}_{0}},E_{0}, h),$$
	where $d_{B}$ is the Berkovits degree \cite{BM} {\it and} $E_{0}$ {\it is any open subset of} $E$ {\it with} $F^{-1}(h)\subset E_{0}$ {\it and} $F$ {\it is bounded on} $\overline{E}_{0}$.
\end{definition}
\section{Main results}\label{Result}
In this section we will discuss the existence of weak solutions of \eqref{P}. For this, we list our assumptions on $f$ associated with our problem to show the existence result.

From new on, we always assume that $\Omega\subset \mathbb{R}^{N} (N\geq 2$) is a bounded domain with a Lipschitz boundary $\partial\Omega,\ p\in C_{+}(\overline{\Omega})$ satisfy the $\log$-H\"{o}lder continuity condition \eqref{i8} with $ 2\leq p^{-}\leq p(x)\leq p^{+}<\infty$ and \\$f$ : $\Omega\times \mathbb{R}\times \mathbb{R}^{N}\rightarrow \mathbb{R}$ is a function such that:\\

$(A_{1})$ $f$ satisfies the Carath\'odory condition.

$(A_{2})$ There exists $\beta_{1}> 0$ and $\gamma\in L^{p'(x)}(\Omega)$ such that
$$
\vert f(x,\zeta, \xi)\vert \leq \beta_{1}(\gamma(x)+\vert \zeta\vert ^{q(x)-1}+\vert \xi\vert ^{q(x)-1})
$$

for a.e. $ x\in\Omega$ and all $(\zeta, \xi)\in \mathbb{R}\times \mathbb{R}^{N}$, where $q\in C_{+}(\overline{\Omega})$ with \\$2\leq q^{-}\leq q(x)\leq q^{+}<p^{-}$.\\
The definition of a weak solutions for problem \eqref{P} can be stated as follows.
\begin{definition}\label{defsolu}
	We call that $u\in W^{1,p(x)}(\Omega)$ is a weak solution of \eqref{P} if
	\begin{equation*}
		\begin{array}{ll}
			\displaystyle  \int_{\Omega}\Big(\vert\nabla u\vert^{p(x)-2}\nabla u+\frac{\vert\nabla u\vert^{2p(x)-2}\nabla u}{\sqrt{1+\vert\nabla u\vert^{2p(x)}}}\Big)\nabla\varphi dx=\int_{\Omega}\lambda f(x, u, \nabla u)\varphi dx,\\
			\mbox{for all} \ \varphi\in W^{1,p(x)}(\Omega).
		\end{array}
	\end{equation*}
\end{definition}
\begin{remark}
	\begin{itemize}
		\item Note that $\displaystyle  \int_{\Omega}\Big(\vert\nabla u\vert^{p(x)-2}\nabla u+\frac{\vert\nabla u\vert^{2p(x)-2}\nabla u}{\sqrt{1+\vert\nabla u\vert^{2p(x)}}}\Big)\nabla\varphi dx$ is well defined (see \cite{Rodri}).
		\item $ \lambda f(x,u,\nabla u)\in L^{p'(x)}(\Omega)$ under $u\in W^{1,p(x)}(\Omega)$ and the given hypotheses about the exponents $p$ and $q$ and assumption $(A_{2})$ because: $\gamma\in L^{p'(x)}(\Omega)$, $r(x)=(q(x)-1)p'(x)\in C_{+}(\overline{\Omega})$ with $r(x)<p(x)$. Then, by Remark \ref{embedding} we can conclude that $L^{p(x)}\hookrightarrow L^{r(x)}$.
		
		Hence, since $\varphi\in L^{p(x)}(\Omega)$, we have $\lambda f(x, u, \nabla u)\varphi\in L^{1}(\Omega)$. This implies that, the integral $\displaystyle \int_{\Omega}\lambda f(x, u, \nabla u)\varphi dx$ exist.
	\end{itemize}
\end{remark}
We are now in the position to get the existence result of weak solution for \eqref{P}.
\begin{theorem}\label{thex}
	If the assumptions $(A_1)-(A_2)$ hold, then the problem \eqref{P} possesses at least one weak solution $u$ in  $W^{1,p(x)}(\Omega)$.
\end{theorem}
Before giving the proof of the Theorem \ref{thex}, we first give two lemmas that will be used later.
	
	Let us consider the following functional:
	\begin{equation*}
		\mathcal{J}(u):=\int_{\Omega}\frac{1}{p(x)}\Big(\vert\nabla u\vert^{p(x)}+\sqrt{1+\vert\nabla u\vert^{2p(x)}}\Big)dx.
	\end{equation*}
	It is obvious that the functional $\mathcal{J}$ is a continuously G\^ateaux differentiable and its G\^ateaux derivative at the point $u\in W^{1,p(x)}(\Omega)$ is the functional $\mathcal{T}:=\mathcal{J}'(u)\in W^{-1,p'(x)}(\Omega)$, given by
	\begin{equation*}
		\displaystyle\langle \mathcal{T}u,\varphi \rangle= \int_{\Omega}\Big(\vert\nabla u\vert^{p(x)-2}\nabla u+\frac{\vert\nabla u\vert^{2p(x)-2}\nabla u}{\sqrt{1+\vert\nabla u\vert^{2p(x)}}}\Big)\nabla\varphi dx, 
	\end{equation*}
	$ \mbox{ for all } \ u,\varphi\in W^{1,p(x)}(\Omega)$ where $\langle\cdot,\cdot \rangle$ the duality pairing between $W^{-1,p'(x)}(\Omega)$ and $W^{1,p(x)}(\Omega)$. Furthermore, the properties of the operator $\mathcal{T}$ are summarized in the following lemma (see \cite[Proposition 3.1.]{Rodri}).
	\begin{lemma}\label{lemma2}The mapping $$\begin{array}{lclllll}
		\mathcal{T}:W^{1,p(x)}(\Omega)\longrightarrow W^{-1,p'(x)}(\Omega) \\
		\displaystyle\langle \mathcal{T}u,\varphi \rangle= \int_{\Omega}\Big(\vert\nabla u\vert^{p(x)-2}\nabla u+\frac{\vert\nabla u\vert^{2p(x)-2}\nabla u}{\sqrt{1+\vert\nabla u\vert^{2p(x)}}}\Big)\nabla\varphi dx, 
		\end{array} $$
		{\it is a continuous, bounded, strictly monotone operator,} and {\it is a mapping of class} $(S_{+})$.
	\end{lemma}
	\begin{lemma}\label{lemma3}
		Assume that the assumptions $(A_1)-(A_2)$ hold. Then the operator\\ $\mathcal{S}:W^{1,p(x)}(\Omega)\rightarrow W^{-1,p'(x)}(\Omega)$ defined by
		$$
		\displaystyle \langle \mathcal{S}u, \varphi\rangle=-\int_{\Omega}\lambda f(x,u, \nabla u)\varphi dx,\ \mbox{ for all }\ u,\varphi\in W^{1,p(x)}(\Omega)
		$$
		is compact.
	\end{lemma}
	\begin{proof}
		Let $\Phi$ : $W^{1,p(x)}(\Omega)\rightarrow L^{p'(x)}(\Omega)$ be an operator defined by
		\begin{equation*}
			\Phi u(x) :=-\lambda f(x,\ u,\ \nabla u) \ \mbox{ for } \ u\in W^{1,p(x)}(\Omega)\ \mbox{ and }\ x\in\Omega.
		\end{equation*}
		Next, we split the proof in several steps.
		\\
		We first show that $\Phi$ is bounded and continuous.
		By using $(A_2)$, the inequalities \eqref{i6} and \eqref{i7}, we obtain
		\begin{equation*}
			\begin{split}
				\displaystyle \vert \Phi u\vert _{p'(x)} &  \leq \rho_{p'(x)}(\Phi u)+1 \\
				& \displaystyle = \int_{\Omega}\vert \lambda f(x, u(x),\nabla u(x))\vert ^{p'(x)}dx+1\\
				& \displaystyle = \int_{\Omega}\vert \lambda\vert^{p'(x)}\vert f(x, u(x),\nabla u(x))\vert ^{p'(x)}dx+1\\
				& \displaystyle \leq\Big(\vert \lambda\vert ^{p'^{-}}+\vert \lambda\vert ^{p'^{+}}\Big)\int_{\Omega}\vert C_{1}\Big(\gamma(x)+\vert u\vert ^{q(x)-1}+\vert \nabla u\vert ^{q(x)-1}\Big)\vert ^{p'(x)}dx+1\\
				& \displaystyle \leq {\it const}\Big(\vert \lambda\vert ^{p'^{-}}+\vert \lambda\vert ^{p'^{+}}\Big)\Big (\rho_{p'(x)}(\gamma)+\rho_{r(x)}(u)+\rho_{r(x)}(\nabla u)\Big)+1\\
				& \displaystyle \leq {\it const} \Big(\vert \gamma\vert _{p(x)}^{p^{\prime+}}+\vert u\vert _{r(x)}^{r^{+}}+\vert u\vert _{r(x)}^{r^{-}}+\vert \nabla u\vert _{r(x)}^{r^{+}}+\vert \nabla u\vert _{r(x)}^{r^{-}}\Big)+1,
			\end{split}
		\end{equation*}
		
		for all $u\in W^{1,p(x)}(\Omega)$.
		\\
		Then, by the continuous embedding $L^{p(x)}\hookrightarrow L^{r(x)}$ and \eqref{i9}, we have
		\begin{equation*}
			\vert \Phi u\vert _{p'(x)}\leq const(\vert \gamma\vert _{p(x)}^{p^{\prime+}}+\vert u\vert _{1,p(x)}^{r^{+}}+\vert u\vert _{1,p(x)}^{r^{-}})+1.
		\end{equation*}
		This implies that $\Phi$ is bounded on $W^{1,p(x)}(\Omega)$. \\
		To show that $\Phi$ is continuous, let $u_{n}\rightarrow u$ in $W^{1,p(x)}(\Omega)$. Then $u_{n}\rightarrow u$ in $L^{p(x)}(\Omega)$ and $\nabla u_{n}\rightarrow\nabla u$ in $(L^{p(x)}(\Omega))^{N}$. Hence there exist a subsequence $(u_{k})$ of $(u_{n})$ and measurable functions $\phi$ in $L^{p(x)}(\Omega)$ and $\psi$ in $(L^{p(x)}(\Omega))^{N}$ such that
		\begin{equation*}
			u_{k}(x)\rightarrow u(x)\ \mbox{ and }\ \nabla u_{k}(x)\rightarrow\nabla u(x),
		\end{equation*}
		\begin{equation}\label{i10}
		\vert u_{k}(x)\vert \leq \phi(x) \ \mbox{ and }\ \vert \nabla u_{k}(x)\vert \leq\vert \psi(x)\vert,
		\end{equation}
		
		for a.e. $ x\in\Omega$ and all $k\in \mathbb{N}$.
		\\
		Hence, thanks to $(A_1)$, we get, as $ k \longrightarrow \infty $
		\begin{equation*}
			f(x,u_{k}(x),\nabla u_{k}(x))\rightarrow f(x,u(x),\nabla u(x))\ \mbox{  a.e. }\ x\in\Omega.
		\end{equation*}
		On the other hand, it follows from $(A_2)$ and \eqref{i10} that
		\begin{equation*}
			\vert f(x, u_{k}(x), \nabla u_{k}(x))\vert \leq \beta_{1}(\gamma(x)+\vert \phi(x)\vert ^{q(x)-1}+\vert \psi(x)\vert ^{q(x)-1}),
		\end{equation*}
		
		for a.e. $ x\in\Omega$ and for all $k\in \mathbb{N}$.
		\\
		Since
		\begin{equation*}
			\gamma+\vert \phi\vert ^{q(x)-1}+\vert \psi(x)\vert ^{q(x)-1}\in L^{p'(x)}(\Omega),
		\end{equation*}
		and
		\begin{equation*}
			\displaystyle \rho_{p'(x)}(\Phi u_{k}-\Phi u)=\int_{\Omega}\vert f(x,u_{k}(x), \nabla u_{k}(x))-f(x, u(x),\nabla u(x))\vert ^{p'(x)}dx,
		\end{equation*}
		therefore, the Lebesgue's theorem and the equivalence \eqref{i5} implies that
		\begin{equation*}
			\displaystyle \Phi u_{k}\rightarrow\Phi u\ \mbox{ in }\ L^{p'(x)}(\Omega).
		\end{equation*}
		Thus the entire sequence $(\Phi u_{n})$ converges to $\Phi u$ in $L^{p'(x)}(\Omega)$.
		\\
		Moreover, let $I^{*}$ : $L^{p'(x)}(\Omega)\rightarrow W^{-1,p'(x)}(\Omega)$ be the adjoint operator for the embedding of\\ $I:W^{1,p(x)}(\Omega)\rightarrow L^{p(x)}(\Omega)$.\\
		We then define
		\begin{equation*}
			\displaystyle  I^{*}o\Phi : W^{1,p(x)}(\Omega)\rightarrow W^{-1,p'(x)}(\Omega),
		\end{equation*}
		
		which is well-defined by assumption $(A_2)$.\\
		Since the embedding $I$ is compact, it is known that the adjoint operator $I^{*}$ is also compact. Therefore, $ I^{*}o\Phi$ is compact. This completes the proof of the Lemma \ref{lemma3}.	
	\end{proof}
	Now we give the proof of the Theorem \ref{thex}. For that, we transform this Neumann boundary value problem into a new one governed by a Hammerstein equation, so by using the theory of topological degree introduced in section \ref{Degree}, we show the existence of a weak solutions to the state problem.
	
	First, for all $u, \varphi\in W^{1,p(x)}(\Omega)$, we define the operators $\mathcal{T}$ and $\mathcal{S}$, as defined in Lemmas \ref{lemma2} and \ref{lemma3} respectively,
	$$\begin{array}{lclllll}
	\mathcal{T}:W^{1,p(x)}(\Omega)\longrightarrow W^{-1,p'(x)}(\Omega) \\
	\displaystyle\langle \mathcal{T}u,\varphi \rangle= \int_{\Omega}\Big(\vert\nabla u\vert^{p(x)-2}\nabla u+\frac{\vert\nabla u\vert^{2p(x)-2}\nabla u}{\sqrt{1+\vert\nabla u\vert^{2p(x)}}}\Big)\nabla\varphi dx, 
	\end{array} $$
	$$\begin{array}{lclllll}
	\mathcal{S}:W^{1,p(x)}(\Omega)\longrightarrow W^{-1,p'(x)}(\Omega)\\
	\displaystyle \langle \mathcal{S}u,\varphi\rangle=-\int_{\Omega}f(x,u, \nabla u)\varphi dx.  
	\end{array} $$
	Then $u\in W^{1,p(x)}(\Omega)$ is a weak solution of \eqref{P} if and only if
	\begin{equation}\label{i11}
	\mathcal{T}u=-\mathcal{S}u.
	\end{equation}
	Thanks to the properties of the operator $\mathcal{T}$ seen in lemma \ref{lemma2} and in view of Minty-Browder Theorem (see \cite[Theorem 26 A]{ref14}, the inverse operator
	\begin{equation*}
		\mathcal{L}:=\mathcal{T}^{-1} : W^{-1,p'(x)}(\Omega)\rightarrow W^{1,p(x)}(\Omega),
	\end{equation*}
	is bounded, continuous and of class $(S_{+})$. Moreover, note by Lemma \ref{lemma3} that the operator $\mathcal{S}$ is bounded, continuous and quasimonotone.
	\\
	Consequently, the equation \eqref{i11} is equivalent to the operator equation
	\begin{equation}\label{i12}
	u=\mathcal{L}\varphi\ \mbox{ and }\ \varphi+\mathcal{S}o\mathcal{L}\varphi=0.
	\end{equation}
	Following Zeidler's terminology \cite{ref14}, the equation $\varphi+\mathcal{S}o\mathcal{L}\varphi=0$ is an abstract Hammerstein equation in the reflexive Banach space $W^{-1,p'(x)}(\Omega)$.
	\\	
	To solve equation \eqref{i12}, we will apply the degree theory introducing in section \ref{Degree}. To do this, set
	\begin{equation*}\label{i13}
		\mathcal{B}:=\Big\{\varphi\in W^{-1,p'(x)}(\Omega):\exists\; t\in[0,1] \ \mbox{ such that  }\ \varphi+t\mathcal{S}o\mathcal{L}\varphi=0 \Big\}.
	\end{equation*}	
	Next,we prove that $ \mathcal{B}$ is bounded in $\in W^{-1,p'(x)}(\Omega)$.\\
	Let $\varphi\in \mathcal{B}$ and set $u:=\mathcal{L}\varphi$, then $\vert \mathcal{L}\varphi\vert _{1,p(x)}=\vert \nabla u\vert _{p(x)}$.	\\
	If $\vert \nabla u\vert _{p(x)}\leq 1$, then $\vert \mathcal{L}\varphi\vert _{1,p(x)}$ is bounded.\\
	If $\vert \nabla u\vert _{p(x)}>1$, then by the implication \eqref{i3}, the growth condition $(A_{2})$, the H\"{o}lder inequality \eqref{i2}, the inequality \eqref{i7} and the Young inequality, we get
	\begin{equation*}
		\begin{split}
			\vert \mathcal{L}\varphi\vert _{1,p(x)}^{p^{-}} &= \vert \nabla u\vert _{p(x)}^{p-}\\
			&\leq\rho_{p(x)}(\nabla u)\\
			&=\langle \varphi,\  \mathcal{L}\varphi\rangle\\
			&= -t\langle \mathcal{S}o\mathcal{L}\varphi,\  \mathcal{L}\varphi\rangle\\
			&=\ t\int_{\Omega}\lambda f(x, u, \nabla u)udx\\
			&\leq {\it const} \Big(\displaystyle \int_{\Omega}\vert \gamma(x)u(x)\vert dx+\rho_{q(x)}(u)+\int_{\Omega}\vert \nabla u\vert ^{q(x)-1}\vert u\vert dx\Big)\\
			&\leq {\it const} \Big(2\vert \gamma\vert _{p'(x)}\vert u\vert _{p(x)}+\vert u\vert _{q(x)}^{q^{+}}+\vert u\vert _{q(x)}^{q^{-}}+\displaystyle \frac{1}{q^{'-}}\rho_{q(x)}(\nabla u)+\frac{1}{q-}\rho_{q(x)}(u)\Big)\\
			&\leq {\it const} \Big(\vert u\vert _{p(x)}+\vert u\vert _{q(x)}^{q^{+}}+\vert u\vert _{q(x)}^{q^{-}}+\vert \nabla u\vert _{q(x)}^{q^{+}}\Big).
		\end{split}
	\end{equation*} 
	From \eqref{i9} and the continuous embedding $L^{p(x)}\hookrightarrow L^{q(x)}$, we conclude that 
	\begin{equation*}
		\vert \mathcal{L}\varphi\vert _{1,p(x)}^{p^{-}}\leq {\it const}\Big(\vert \mathcal{L}\varphi\vert _{1,p(x)}+\vert \mathcal{L}\varphi\vert _{1,p(x)}^{q^{+}}\Big).
	\end{equation*}
	So, we infer that $\Big\{\mathcal{L}\varphi\vert \varphi\in \mathcal{B}\Big\}$ is bounded.\\
	Since the operator $\mathcal{S}$ is bounded, it is obvious from \eqref{i12} that $\mathcal{B}$ is bounded in $W^{-1,p'(x)}(\Omega)$. Consequently, there exists $R >0$ such that
	\begin{equation*}
		\displaystyle\vert \varphi\vert _{-1,p'(x)}<R\ \mbox{ for all } \ \varphi\in \mathcal{B}.
	\end{equation*}
	Therefore
	\begin{equation*}
		\varphi+t\mathcal{S}o\mathcal{L}\varphi\neq 0\ \mbox{ for all } \ \varphi\in\partial B_{R}(0)\ \mbox{ and all } \ t\in[0,1],
	\end{equation*}
	
	where $B_{R}(0)$ is the ball of center $0$ and radius $R$ in $W^{-1,p'(x)}(\Omega)$.\\
	Moreover, from Lemma \ref{lemma1} we conclude that
	\begin{equation*}
		I+\mathcal{S}o\mathcal{L}\in \mathcal{F}_{\mathcal{L}}(\overline{\mathrm{\mathcal{B}}_{\mathrm{R}}(0)})\ \mbox{ and } \ I=\mathcal{T}o\mathcal{L}\in \mathcal{F}_{\mathcal{L}}(\overline{\mathrm{\mathcal{B}}_{\mathrm{R}}(0)}).
	\end{equation*}
	Since the operators $I$, $\mathcal{S}$ and $\mathcal{L}$ are bounded, $I+\mathcal{S}o\mathcal{L}$ is also bounded. So, we
	conclude that
	
	\begin{equation*}
		I+\mathcal{S}o\mathcal{L}\in \mathcal{F}_{\mathcal{L},B}(\overline{\mathrm{\mathcal{B}}_{\mathrm{R}}(0)})\ \mbox{ and } \ I=\mathcal{T}o\mathcal{L}\in \mathcal{F}_{\mathcal{L},B}(\overline{\mathrm{\mathcal{B}}_{\mathrm{R}}(0)}).
	\end{equation*}
	Next, consider a homotopy $\Lambda$ : $[0,1]\times\overline{B_{R}(0)}\rightarrow W^{-1,p'(x)}(\Omega)$ given by
	\begin{equation*}
		\Lambda(t,\varphi) :=\varphi+t\mathcal{S}o\mathcal{L}\varphi \ \mbox{ for }\ (t, \varphi)\in[0,1]\times\overline{B_{R}(0)}.
	\end{equation*}
	Hence, by using  the normalization property and the homotopy invariance of degree $d$, we obtain
	\begin{equation*}
		d(I+\mathcal{S}o\mathcal{L}, B_{R}(0), 0)=d(I, B_{R}(0),\ 0)=1\neq 0.
	\end{equation*}
	Then, there exists a point $\varphi\in B_{R}(0)$ such that
	\begin{equation*}
		\varphi+\mathcal{S}o\mathcal{L}\varphi=0.
	\end{equation*}
	Thus, we conclude that $u=\mathcal{L}\varphi$ is a weak solution of \eqref{P}. The proof is completed.

\end{document}